\documentclass{amsart}
\usepackage{amscd,amsmath,amssymb}
\usepackage{pstricks}
\usepackage{latexsym,amsbsy,mathrsfs}
\usepackage{xy}
\usepackage{xypic}
\usepackage{pgf,tikz}

\newtheorem{thm}{Theorem}[section]
\newtheorem{lem}[thm]{Lemma}
\newtheorem{cor}[thm]{Corollary}
\newtheorem{prop}[thm]{Proposition}

\theoremstyle{definition}
\newtheorem{example}[thm]{Example}

\newtheorem{defn}[thm]{Definition}

\newtheorem{rem}[thm]{Remark}
\newtheorem{ass}[thm]{Assumption}

\numberwithin{equation}{thm}

\begin{document}
\title[ $n$-ANGULATED CATEGORIES FROM SELF-INJECTIVE ALGEBRAS ]
{$n$-ANGULATED CATEGORIES FROM SELF-INJECTIVE ALGEBRAS}

\author{Zengqiang Lin}
\address{ School of Mathematical sciences, Huaqiao University,
Quanzhou\quad 362021,  China.} \email{lzq134@163.com}

\thanks{This work was supported  by  the Natural Science Foundation of Fujian Province (Grants No. 2013J05009) and the Science Foundation of Huaqiao University (Grants No. 2014KJTD14)}



\begin{abstract}
Let $\mathcal{C}$ be a $k$-linear category with split idempotents, and $\Sigma:\mathcal{C}\rightarrow\mathcal{C}$ an automorphism.
We show that there is an $n$-angulated structure on $(\mathcal{C},\Sigma)$ under certain conditions.
As an application, we obtain a class of examples of $n$-angulated categories from self-injective algebras.
\end{abstract}

\maketitle

\section{Introduction}
Let $n$ be an integer greater than or equal to three. Geiss, Keller and Oppermann  introduced the notion of $n$-angulated categories, which is a $``$higher dimensional" analogue of triangulated categories, and gave the standard construction of $n$-angulated categories from  $(n-2)$-cluster tilting subcategories of a triangulated category which are closed under the $(n-2)$-nd power of the suspension functor \cite{[GKO]}. For $n=3$, an $n$-angulated category is nothing but a  classical triangulated category. Another examples of $n$-angulated categories from local algebras were given in \cite{[BT2]}. Let $R$ be a commutative local ring with maximal principal ideal ${\bf{m}}=(p)$ satisfying ${\bf{m}}^2=0$. Then the category of finitely generated free $R$-modules has a structure of $n$-angulation whenever $n$ is even, or when $n$ is odd and $2p=0$ in $R$. The theory of $n$-angulated categories has been developed further, we can see \cite{[BT1],[BT3],[J],[Jo],[L1],[L2]} for reference.  In this note, we devote to provide new class of examples of $n$-angulated categories.

Throughout this paper let $k$ be an algebraically closed field, and let $\mathcal{C}$ be a $k$-linear category with split idempotents and $\Sigma:\mathcal{C}\rightarrow\mathcal{C}$ an automorphism. It is natural to ask under which conditions does the category $(\mathcal{C}, \Sigma)$ has an $n$-angulation. We first note that if $\mathcal{C}$ is an $n$-angulated category, then the category $\mbox{mod}\mathcal{C}$ of contravariant finitely presented and exact  functors from $\mathcal{C}$ to mod$k$ is a Frobenius category. Now we assume that $\mbox{mod}\mathcal{C}$ is a Frobenius category. Then the stable category $\underline{\mbox{mod}}\mathcal{C}$ is a triangulated category and the suspension  is the cosyzygy functor $\Omega^{-1}$. The automorphism $\Sigma$ can be extended to an exact functor from $\mbox{mod}\mathcal{C}$ to $\mbox{mod}\mathcal{C}$ and thus to a triangle functor of  $\underline{\mbox{mod}}\mathcal{C}$.  In this case, $(\Sigma, \sigma)$ and $(\Omega^{-n},(-1)^n)$ are two triangle endofunctors of $\underline{\mbox{mod}}\mathcal{C}$, where $\sigma:\Sigma\Omega^{-1}\rightarrow\Omega^{-1}\Sigma$ is a natural isomorphism. Heller showed in \cite{[He]} that there is a bijection between the class of pre-triangulations of $(\mathcal{C}, \Sigma)$ and the class of  isomorphisms of triangle functors from $(\Sigma, \sigma)$ to $(\Omega^{-3},-1)$. Since Heller did not succeed in proving the octahedral axiom, Amiot gave a necessary condition on the functor $\Sigma$ such  that $(\mathcal{C},\Sigma)$ has a triangulated structure, which is applied to deformed preprojective algebras \cite{[Am]}. In \cite{[GKO]} it is showed that Heller's parametrization of pre-triangulations extends to pre-$n$-angulations.
Our first main result is as follows.

\begin{thm}\label{1} Let $\mathcal{C}$ be a $k$-linear category with split idempotents and $\Sigma:\mathcal{C}\rightarrow\mathcal{C}$ an automorphism. If $\mbox{mod}\mathcal{C}$ is a Frobenius category and there exists an exact sequence of exact endofunctors of $\mbox{mod}\mathcal{C}$
$$0\rightarrow \mbox{Id}\rightarrow  X^{1}\rightarrow X^{2}\rightarrow\cdots\rightarrow X^{n}\rightarrow \Sigma\rightarrow 0, $$
where  all the $X^{i}$ take values in $\mbox{proj}\mathcal{C}$. Then $(\mathcal{C},\Sigma)$ has an $n$-angulation structure.
\end{thm}

Theorem \ref{1} is a higher version of  \cite[Theorem 8.1]{[Am]}. Since $n$-angulated categories are more complex than triangulated categories,  we should make some  technological modifications in the proof.

Let $A$ be a finite-dimensional $k$-algebra. Given an automorphism $\sigma$ of $A$, we denote by $_1A_\sigma$ the bimodule structure on $A$ where the action on right is twisted by $\sigma$.
It is easy to check that $_1A_\sigma\otimes_A$ $_1A_\tau\cong$ $_1A_{\tau\sigma}$, where $\sigma$ and $\tau$ are two automorphisms.
A finite-dimensional $k$-algebra $A$ is said to be $quasi$-$periodic$ if  $A$ has a quasi-periodic projective resolution over the enveloping algebra $A^e=A^{\tiny\mbox{op}}\otimes_k A$, i.e., $\Omega^n_{A^e}(A)\cong\ _1A_\sigma$ as $A$-$A$-bimodules for some natural number $n$ and some automorphism $\sigma$ of $A$. In particular, $A$ is  $periodic$ if $\Omega^n_{A^e}(A)\cong A$ as bimodules.  In this case, if $n$ is minimal, we say $A$ is a periodic algebra of $periodicity$ $n$. It is well known that quasi-periodic algebras are self-injective algebras \cite{[GSS]}.
Our second main result is as follows.


\begin{thm}(=Theorem 4.3.)\label{2}
Let $A$ be a finite-dimensional indecomposable quasi-periodic $k$-algebra. Assume that $\Omega^n_{A^e}(A)\cong$ $_1A_\sigma$ as $A$-$A$-bimodules for
  an automorphism $\sigma$ of $A$. Then for each positive integer $m$,  the category $(\mbox{proj}A, \Sigma)$ has an $mn$-angulation structure, where $\Sigma$ is the functor $-\otimes_AA_{\sigma^{-m}}: \mbox{proj}A\rightarrow\mbox{proj}A$.
  In particular, if $\sigma$ is of finite order $l$, then $(\mbox{proj}A, \mbox{Id}_{\mbox{proj}A})$ has an $ln$-angulation structure.
\end{thm}


 There are numerous examples of periodic algebras. The most notable examples are preprojective algebras of Dynkin graphs, whose periodicity at most 6. These results have been generalized to deformed preprojective algebras \cite{[BES]}. Dugas showed that each self-injective algebra of finite representation type is periodic \cite{[D1]}. We also can obtain periodic algebras as endomorphism algebras of periodic $d$-cluster-tilting objects in a triangulated category \cite{[D2]}.
Therefore, by Theorem \ref{2} we can construct a large class of examples of $n$-angulated categories from self-injective algebras.

This paper is organized as follows. In Section 2, we  recall the definition of $n$-angulated category and make some preliminaries to prove our first main result. In Section 3, we prove Theorem \ref{1}. In Section 4, we prove Theorem \ref{2} and give some examples.


\section{Definitions and preliminaries}

Let $\mathcal{C}$ be an additive category equipped with an automorphism $\Sigma:\mathcal{C}\rightarrow\mathcal{C}$. 
 An $n$-$\Sigma$-$sequence$ in $\mathcal{C}$ is a sequence of morphisms
$$X_\bullet= (X_1\xrightarrow{f_1}X_2\xrightarrow{f_2}X_3\xrightarrow{f_3}\cdots\xrightarrow{f_{n-1}}X_n\xrightarrow{f_n}\Sigma X_1).$$
Its $left\ rotation$ is the $n$-$\Sigma$-sequence
$$X_2\xrightarrow{f_2}X_3\xrightarrow{f_3}X_4\xrightarrow{f_4}\cdots\xrightarrow{f_{n-1}}X_n\xrightarrow{f_n}\Sigma X_1\xrightarrow{(-1)^n\Sigma f_1}\Sigma X_2.$$  We can define $right\ rotation\ of\ an$ $n$-$\Sigma$-$sequence$ similarly.
An $n$-$\Sigma$-$sequence$ $X_\bullet$ is $exact$ if the induced sequence
$$\cdots\rightarrow \mathcal{C}(-,X_1)\rightarrow \mathcal{C}(-,X_2)\rightarrow\cdots\rightarrow \mathcal{C}(-,X_n)\rightarrow \mathcal{C}(-,\Sigma X_1)\rightarrow\cdots$$
is exact. A $morphism\ of$ $n$-$\Sigma$-$sequences$ is  a sequence of morphisms $\varphi_\bullet=(\varphi_1,\varphi_2,\cdots,\varphi_n)$ such that the following diagram commutes
$$\xymatrix{
X_1 \ar[r]^{f_1}\ar[d]^{\varphi_1} & X_2 \ar[r]^{f_2}\ar[d]^{\varphi_2} & X_3 \ar[r]^{f_3}\ar[d]^{\varphi_3} & \cdots \ar[r]^{f_{n-1}}& X_n \ar[r]^{f_n}\ar[d]^{\varphi_n} & \Sigma X_1 \ar[d]^{\Sigma \varphi_1}\\
Y_1 \ar[r]^{g_1} & Y_2 \ar[r]^{g_2} & Y_3 \ar[r]^{g_3} & \cdots \ar[r]^{g_{n-1}} & Y_n \ar[r]^{g_n}& \Sigma Y_1\\
}$$
where each row is an $n$-$\Sigma$-sequence. It is an {\em isomorphism} if $\varphi_1, \varphi_2, \cdots, \varphi_n$ are all isomorphisms in $\mathcal{C}$.

\begin{defn} (\cite{[GKO]})
An $n$-$angulated\ category$ is a triple $(\mathcal{C}, \Sigma, \Theta)$, where $\mathcal{C}$ is an additive category, $\Sigma$ is an automorphism of $\mathcal{C}$, and $\Theta$ is a class of $n$-$\Sigma$-sequences satisfying the following axioms:

(N1) (a) The class $\Theta$ is closed under direct sums and direct summands.

(b) For each object $X\in\mathcal{C}$ the trivial sequence
$$X\xrightarrow{1_X}X\rightarrow 0\rightarrow\cdots\rightarrow 0\rightarrow \Sigma X$$
belongs to $\Theta$.

(c) For each morphism $f_1:X_1\rightarrow X_2$ in $\mathcal{C}$, there exists an $n$-$\Sigma$-sequence in $\Theta$ whose first morphism is $f_1$.

(N2) An $n$-$\Sigma$-sequence belongs to $\Theta$ if and only if its left rotation belongs to $\Theta$.

(N3) Each commutative diagram
$$\xymatrix{
X_1 \ar[r]^{f_1}\ar[d]^{\varphi_1} & X_2 \ar[r]^{f_2}\ar[d]^{\varphi_2} & X_3 \ar[r]^{f_3}\ar@{-->}[d]^{\varphi_3} & \cdots \ar[r]^{f_{n-1}}& X_n \ar[r]^{f_n}\ar@{-->}[d]^{\varphi_n} & \Sigma X_1 \ar[d]^{\Sigma \varphi_1}\\
Y_1 \ar[r]^{g_1} & Y_2 \ar[r]^{g_2} & Y_3 \ar[r]^{g_3} & \cdots \ar[r]^{g_{n-1}} & Y_n \ar[r]^{g_n}& \Sigma Y_1\\
}$$ with rows in $\Theta$ can be completed to a morphism of  $n$-$\Sigma$-sequences.

(N4) In the situation of (N3), the morphisms $\varphi_3, \varphi_4, \cdots,\varphi_n$ can be chosen such that the mapping cone
$$X_2\oplus Y_1\xrightarrow{\left(
                              \begin{smallmatrix}
                                -f_2 & 0 \\
                                \varphi_2 & g_1 \\
                              \end{smallmatrix}
                            \right)}
 X_3\oplus Y_2 \xrightarrow{\left(
                              \begin{smallmatrix}
                                -f_3 & 0 \\
                                \varphi_3 & g_2 \\
                              \end{smallmatrix}
                            \right)}
 \cdots \xrightarrow{\left(
                            \begin{smallmatrix}
                               -f_n & 0 \\
                                \varphi_n & g_{n-1} \\
                             \end{smallmatrix}
                           \right)}
 \Sigma X_1\oplus Y_n \xrightarrow{\left(
                              \begin{smallmatrix}
                                -\Sigma f_1 & 0 \\
                                \Sigma\varphi_1 & g_n \\
                              \end{smallmatrix}
                            \right)}
 \Sigma X_2\oplus \Sigma Y_1 \\
$$
belongs to $\Theta$.
\end{defn}

\begin{rem}
(a) If $(\mathcal{C}, \Sigma, \Theta)$ is an $n$-angulated category, then $\Sigma$ is called a $suspension\ functor$ and $\Theta$ is called an $n$-$angulation$ of $(\mathcal{C}, \Sigma)$ whose elements are called $n$-$angles$. If $\Theta$ only satisfies the three axioms (N1),(N2) and (N3), then  $\Theta$ is called a $pre$-$n$-$angulation$ of $(\mathcal{C}, \Sigma)$ and the triple $(\mathcal{C}, \Sigma, \Theta)$ is called a $pre$-$n$-$angulated\ category$. In this case, an element of $\Theta$ is also called an $n$-$angle$.

(b) An $n$-$\Sigma$-$complex$ is a complex $X_\bullet=(X_i,f_i)_{i\in\mathbb{Z}}$ over $\mathcal{C}$ such that $X_{k+n}=\Sigma X_k$ and $f_{k+n}=\Sigma f_k$ for all $k\in\mathbb{Z}$. Let $X_\bullet= (X_1\xrightarrow{f_1}X_2\xrightarrow{f_2}X_3\xrightarrow{f_3}\cdots\xrightarrow{f_{n-1}}X_n\xrightarrow{f_n}\Sigma X_1)$ be an $n$-angle in a pre-$n$-angulated category $\mathcal{C}$, then $X_\bullet$ is exact by \cite[Proposition 2.5(a)]{[GKO]}, which implies that the compositions $f_2f_1,f_3f_2,\cdots, f_nf_{n-1},\Sigma f_1\cdot f_n$ are all zero morphisms. So $X_\bullet$ can be naturally seen as an $n$-$\Sigma$-complex.

(c) The automorphism $\Sigma$ of $\mathcal{C}$ induces an exact functor from $\mbox{mod}\mathcal{C}$ to $\mbox{mod}\mathcal{C}$ defined by $M\mapsto M\cdot \Sigma^{-1}$, which is also denoted by $\Sigma$.
\end{rem}

Throughout this section, we make the following assumption.
\begin{ass}
Let $\mathcal{C}$ be a $k$-linear category with split idempotents and $\Sigma:\mathcal{C}\rightarrow\mathcal{C}$ an automorphism, which satisfying the following two conditions:

(1) The category $\mbox{mod}\mathcal{C}$ is a Frobenius category.

(2) There exists an exact sequence of exact endofunctors of $\mbox{mod}\mathcal{C}$
$$\begin{gathered}0\rightarrow \mbox{Id}\rightarrow X^{1}\rightarrow X^{2}\rightarrow\cdots\rightarrow X^{n}\rightarrow \Sigma\rightarrow 0 \end{gathered}\eqno(2.1) $$
where all the $X^{i}$ take values in $\mbox{proj}\mathcal{C}$.
\end{ass}

Since $\mathcal{C}$ has split idempotents, the Yoneda functor gives a natural equivalence between $\mathcal{C}$ and $\mbox{proj}\mathcal{C}$, which is the subcategory of $\mbox{mod}\mathcal{C}$ consisting of projectives. For convenience we identify $\mathcal{C}$ with $\mbox{proj}\mathcal{C}$. 
Since $\mbox{mod}\mathcal{C}$ is a Frobenius category, we get $\mbox{proj}\mathcal{C}=\mbox{inj}\mathcal{C}$ and the quotient category $\underline{\mbox{mod}}\mathcal{C}$ is a triangulated category with the suspension functor $\Omega^{-1}$. In this case, the automorphism $\Sigma$ of $\mathcal{C}$ in fact induces a triangle functor of $\underline{\mbox{mod}}\mathcal{C}$ which is also denoted by $\Sigma$.
For each $M\in\mbox{mod}\mathcal{C}$, we fix a short exact sequence $0\rightarrow M\rightarrow I_M\rightarrow\Omega^{-1}M\rightarrow0$ with $I_M\in\mathcal{C}$. Thus we obtain a standard injective resolution
$$\begin{gathered}I_M\rightarrow I_{\Omega^{-1}M}\rightarrow I_{\Omega^{-2}M}\rightarrow\cdots \end{gathered}\eqno(2.2)$$
of $M$.

\begin{lem}
There exists a functorial isomorphism $\alpha:\Sigma\rightarrow\Omega^{-n}$.
\end{lem}

\begin{proof}
For each $M\in\mbox{mod}\mathcal{C}$, by (2.1) and (2.2) we obtain the following commutative diagram
$$\xymatrix{
0 \ar[r] & M \ar[r]\ar@{=}[d] & X^1M \ar[r]\ar[d] & X^2M \ar[r]\ar[d] & \cdots \ar[r] & X^{n}M \ar[r]\ar[d] & \Sigma M \ar[d]^{\alpha_M} \ar[r]& 0\\
0 \ar[r] & M \ar[r] & I_M \ar[r] & I_{\Omega^{-1}M} \ar[r] & \cdots \ar[r] & I_{\Omega^{1-n}M} \ar[r] & \Omega^{-n}M \ar[r] & 0\\
}$$
with exact rows, which implies that $\alpha_M:\Sigma M\xrightarrow{\sim}\Omega^{-n}M$ in $\underline{\mbox{mod}}\mathcal{C}$. For each morphism $f:M\rightarrow M'$ in $\mbox{mod}\mathcal{C}$, we can easily deduce that $\Omega^{-n}f\cdot\alpha_M=\alpha_{M'}\cdot\Sigma f$ by comparison theorem.
\end{proof}

Let $$\begin{gathered}X_1\xrightarrow{f_1}X_2\xrightarrow{f_2}X_3\xrightarrow{f_3}\cdots\xrightarrow{f_{n-1}}X_n\xrightarrow{f_n}\Sigma X_1\end{gathered}\eqno(2.3)$$
be an exact $n$-$\Sigma$-sequence in $\mathcal{C}$, and $M=\mbox{ker}f_1$. We note that (2.3) can be seen as the beginning of an injective resolution of $M$. Since $f_n$ has a factorization $X_n\twoheadrightarrow \Sigma M\rightarrowtail \Sigma X_1$, there exists an isomorphism $\beta_M:\Sigma M\xrightarrow{\sim}\Omega^{-n}M$ in $\underline{\mbox{mod}}\mathcal{C}$.

\begin{defn}
Denote by $\Phi$ the class of exact $n$-$\Sigma$-sequences $$X_1\xrightarrow{f_1}X_2\xrightarrow{f_2}X_3\xrightarrow{f_3}\cdots\xrightarrow{f_{n-1}}X_n\xrightarrow{f_n}\Sigma X_1$$
in $\mathcal{C}$ such that $\beta_{\mbox{ker}f_1}=\alpha_{\mbox{ker}f_1}$.
\end{defn}

We will show in next section that $(\mathcal{C},\Sigma,\Phi)$ is an $n$-angulated category. In the rest of this section, we will give an effective description on the elements in $\Phi$, see Proposition \ref{2.5}  for detail.

For each $M\in\mbox{mod}\mathcal{C}$, we denote by $T_M$ the $n$-$\Sigma$-sequence
$$X^{1}M\rightarrow X^{2}M\rightarrow\cdots\rightarrow X^{n}M\rightarrow \Sigma X^1M$$
induced by the exact sequence (2.1). It is easy to see that $T_M\in\Phi$.
We call $T_M$ a {\em standard $n$-angle}.

We denote by $C^{\tiny\mbox{ex}}(\mbox{proj}\mathcal{C})$ the category of acyclic complexes over $\mbox{proj}\mathcal{C}$. Denote by $C^{\tiny\mbox{ex}}_{n\mbox{-}\Sigma}(\mbox{proj}\mathcal{C})$ the non-full subcategory of $C^{\tiny\mbox{ex}}(\mbox{proj}\mathcal{C})$ whose objects are acyclic $n$-$\Sigma$-complexes of the following form
$$(X_\bullet, f_\bullet)=\cdots\rightarrow X_1\xrightarrow{f_1}X_2\xrightarrow{f_2}X_3\xrightarrow{f_3}\cdots\xrightarrow{f_{n-1}}X_n\xrightarrow{f_n}\Sigma X_1\xrightarrow{\Sigma f_1} \Sigma X_2\rightarrow\cdots,$$
and whose morphisms are $n$-$\Sigma$-periodic. We note that
the category $C^{\tiny\mbox{ex}}_{n\mbox{-}\Sigma}(\mbox{proj}\mathcal{C})$ is a Frobenius category and the projective-injectives are the $n$-$\Sigma$-contractible complexes, i.e., the complexes homotopic to zero with an $n$-$\Sigma$-periodic homotopy. The functor $Z_1:C^{\tiny\mbox{ex}}_{n\mbox{-}\Sigma}(\mbox{proj}\mathcal{C})\rightarrow \mbox{mod}\mathcal{C}$ which sends a complex $(X_\bullet,f_\bullet)$ to ker$f_1$ and the functor $T:\mbox{mod}\mathcal{C}\rightarrow C^{\tiny\mbox{ex}}_{n\mbox{-}\Sigma}(\mbox{proj}\mathcal{C})$ which sends an object $M$ to $T_M$ are exact functors. Both of the two functors preserve the projective-injectives, thus we get the following lemma.

\begin{lem}\label{2.1}
The functors $Z_1$ and $T$ induce triangle functors $Z_1:K^{\tiny\mbox{ex}}_{n\mbox{-}\Sigma}(\mbox{proj}\mathcal{C})\rightarrow \underline{\mbox{mod}}\mathcal{C}$ and
$T:\underline{\mbox{mod}}\mathcal{C}\rightarrow K^{\tiny\mbox{ex}}_{n\mbox{-}\Sigma}(\mbox{proj}\mathcal{C})$. Moreover, $Z_1T= Id_{\underline{\mbox{mod}}\mathcal{C}}$.
\end{lem}

An object in $C^{\tiny\mbox{ex}}_{n\mbox{-}\Sigma}(\mbox{proj}\mathcal{C})$ is simply called an $n$-$\Sigma$-$complex$ in $\mathcal{C}$. Since an exact $n$-$\Sigma$-sequences in $\mathcal{C}$ can naturally extend to an $n$-$\Sigma$-complex, we can view $\Phi$ as a full subcategory of $C^{\tiny\mbox{ex}}_{n\mbox{-}\Sigma}(\mbox{proj}\mathcal{C})$. In this sense an object in $\Phi$ is called a $\Phi$-$n$-$\Sigma$-$complex$.


\begin{lem}\label{2.2}
Let $(X_\bullet, f_\bullet)$ 
 be an  $n$-$\Sigma$-complex and $(Y_\bullet, g_\bullet)$
a $\Phi$-$n$-$\Sigma$-complex. If $\varphi_\bullet: X_\bullet\rightarrow Y_\bullet$  is homotopy-equivalent, then $X_\bullet$ is also a $\Phi$-$n$-$\Sigma$-complex.
\end{lem}

\begin{proof}
Let $M=\mbox{ker}f_1$ and $N=\mbox{ker}g_1$, then $\varphi_\bullet$ induces a morphism $h=Z_1(\varphi_\bullet):M\rightarrow N$. The morphism $h$ is an isomorphism in $\underline{\mbox{mod}}\mathcal{C}$ since $\varphi_\bullet$ is an isomorphism in $K^{\tiny\mbox{ex}}_{n\mbox{-}\Sigma}(\mbox{proj}\mathcal{C})$. By comparison theorem we have $\Omega^{-n}h\cdot\beta_M=\beta_N\cdot\Sigma h$. We also have $\Omega^{-n}h\cdot\alpha_M=\alpha_N\cdot\Sigma h$ by the naturality of $\alpha$.  Since $\beta_N=\alpha_N$, we obtain that $\beta_M=(\Omega^{-n}h)^{-1}\cdot\alpha_N\cdot\Sigma h=\alpha_M$, so $X_\bullet$ is a $\Phi$-$n$-$\Sigma$-complex.
\end{proof}

\begin{lem}\label{2.3}
Each commutative diagram $$\xymatrix{
 X_1 \ar[r]^{f_1}\ar[d]^{\varphi_1} & X_2 \ar[r]^{f_2}\ar[d]^{\varphi_2} & X_3 \ar[r]^{f_3} & \cdots \ar[r]^{f_{n-1}}& X_n \ar[r]^{f_n} & \Sigma X_1 \ar[d]^{\Sigma \varphi_1}\\
 Y_1 \ar[r]^{g_1} & Y_2 \ar[r]^{g_2} & Y_3 \ar[r]^{g_3} & \cdots \ar[r]^{g_{n-1}} & Y_n \ar[r]^{g_n}& \Sigma Y_1\\
}$$ whose rows are $\Phi$-$n$-$\Sigma$-complexes can be extended to an $n$-$\Sigma$-periodic morphism.
\end{lem}

\begin{proof}
Since the $Y_i$'s are projective-injectives, by the factorization property of cokernel and the definition of injective we can find morphisms $\varphi_i:X_i\rightarrow Y_i$ such that $\varphi_{i}f_{i-1}=g_{i-1}\varphi_{i-1}$, where $i=3,4,\cdots,n$. Let $M=\mbox{ker}f_1$, $N=\mbox{ker}g_1$, and $h:M\rightarrow N$ be the morphism induced by the left commutative square.
Assume that $f_n$ has a factorization $l_n\pi_n:X_n\twoheadrightarrow\Sigma M\rightarrowtail\Sigma X_1$ and $g_n$ has a factorization $l'_n\pi'_n:Y_n\twoheadrightarrow\Sigma N\rightarrowtail\Sigma Y_1$. The morphism $\varphi_n$ induces a morphism $p:\Sigma M\rightarrow\Sigma N$ such that $p\pi_n=\pi_n'\varphi_n$. It should be noted that we do not have $\Sigma\varphi_1\cdot l_n=l'_n p$, but we have $\Sigma\varphi_1\cdot l_n=l'_n\cdot \Sigma h$.
$$\xymatrix{
X_n \ar[ddd]^{\varphi_n} \ar[rrr]^{f_n} \ar@{->>}[rrd]^{\pi_n}& & & \Sigma X_1 \ar[ddd]^{\Sigma \varphi_1}\\
& &\Sigma M \ar[ddd]^{p} \ar@{>->}[ur]^{l_n} \ar@{-->}[ld]^{a}  & \\
& I \ar@{-->}[ld]^{c} \ar@{-->}[rdd]^{b} &  & \\
Y_n \ar[rrr]_{g_n}\ar@{->>}[rrd]^{\pi'_n}& & & \Sigma Y_1\\
& & \Sigma N \ar@{>->}[ur]^{l_n'} & \\
}$$
Note that $\beta_N\cdot p=\Omega^{-n}h\cdot\beta_M$ by comparison theorem. On the other hand, we have $\alpha_N\cdot \Sigma h=\Omega^{-n}h\cdot\alpha_M$ by the naturality of $\alpha$. Since $\beta_M=\alpha_M$ and $\beta_N=\alpha_N$, we obtain that $p=\alpha_N^{-1}\cdot\Omega^{-n}h\cdot\alpha_M=\Sigma h$ in $\underline{\mbox{mod}}\mathcal{C}$. Thus there exists a projective-injective $I$ in $\mbox{mod}\mathcal{C}$ and morphisms $a:\Sigma M\rightarrow I$ and $b:I\rightarrow\Sigma N$ such that $p-\Sigma h=ba$. As $I$ is projective, there exists a morphism $c:I\rightarrow Y_n$ such that $b=\pi_n'c$. We put $\varphi_n'=\varphi_n-ca\pi_n$, then $\varphi_n'f_{n-1}=\varphi_nf_{n-1}=g_{n-1}\varphi_{n-1}$ and $g_n\varphi_n'=l_n'\pi_n'(\varphi_n-ca\pi_n)=l_n'(p-ba)\pi_n=l_n'\cdot\Sigma h\cdot \pi_n=\Sigma\varphi_1\cdot l_n\pi_n=\Sigma\varphi_1\cdot f_n$.  Thus $(\varphi_1,\varphi_2,\varphi_3,\cdots,\varphi_{n-1},\varphi_n')$ is a morphism in $C^{ex}_{n\mbox{-}\Sigma}(\mbox{proj}\mathcal{C})$.
\end{proof}

\begin{lem} \label{2.4}
The functor $Z_1:K^{ex}_{n\mbox{-}\Sigma}(\mbox{proj}\mathcal{C})\rightarrow \underline{\mbox{mod}}\mathcal{C}$ is full and its kernel  is an ideal whose square vanishes. Thus $Z_1$ detects isomorphisms, that is, if $Z_1(f)$ is an isomorphism in $\underline{\mbox{mod}}\mathcal{C}$, then $f$ is a homotopy-equivalence.
\end{lem}

\begin{proof}
The last assertion follows from the first assertion and \cite[Lemma 8.6]{[Am]}. By Lemma \ref{2.1}, we have $Z_1T=\mbox{Id}_{\underline{\mbox{mod}}\mathcal{C}}$, which implies  that $Z_1$ is full. We only need to show that ker$Z_1$ is an ideal whose square vanishes.

Let $\varphi_\bullet:(X_\bullet,f_\bullet)\rightarrow (Y_\bullet,g_\bullet)$ be a morphism of $n$-$\Sigma$-complexes with $Z_1(\underline{\varphi_\bullet})=0$. Let $(M,l:M\rightarrow X_1)$ be the kernel of $f_1$ and $\Sigma^{-1}f_n=l\pi$. Similarly let $(N,l':N\rightarrow Y_1)$ be the kernel of $g_1$ and $\Sigma^{-1}g_n=l'\pi'$. Then $h=Z_1(\varphi_\bullet)$ has a factorization $M\xrightarrow{a}I\xrightarrow{b}N$, where $I$ is projective-injective. Thus there exist two morphisms $c:X_1\rightarrow I$ and $d:I\rightarrow\Sigma^{-1}Y_n$ such that $a=cl$ and $b=\pi'd$. Let $h_1=dc$. Note that $(\varphi_1-\Sigma^{-1}g_n\cdot h_1)\Sigma^{-1}f_n=0$, there exists a morphism $m_1:M_1\rightarrow Y_1$ such that $\varphi_1-\Sigma^{-1}g_n\cdot h_1=m_1\pi_1$. Since $Y_1$ is projective-injective, there exists a morphism $h_2:X_2\rightarrow Y_1$ such that $m_1=h_2l_1$. Thus $\varphi_1=\Sigma^{-1}g_n\cdot h_1+h_2f_1$.
Similarly we can show that there exist morphisms $h_{i+1}:X_{i+1}\rightarrow Y_{i}$ such that $\varphi_i=h_{i+1}f_i+g_{i-1}h_i$, $i=2,3,\cdots,n$.
We take $\varphi_n'=(h_{n+1}-\Sigma h_1)f_n$, then $\varphi_n-\varphi_n'=g_{n-1}h_n+\Sigma h_1\cdot f_n$. Hence the morphism $\varphi_\bullet$ is homotopy to the morphism $\varphi'_\bullet=(0,0,\cdots,0,\varphi_n')$ with an $n$-$\Sigma$-periodic homotopy $(h_1, h_2, \cdots, h_n)$.
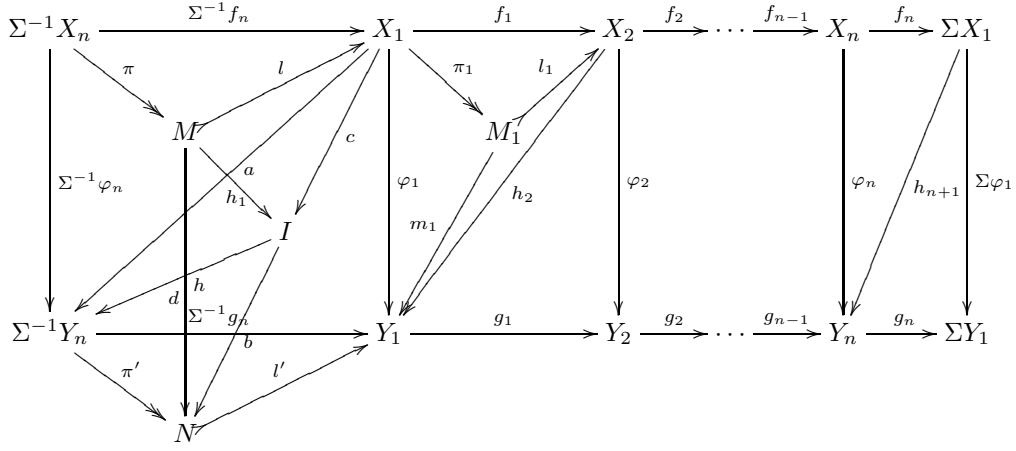
$$\xymatrix{
\Sigma^{-1}X_n \ar[ddd]^{\Sigma^{-1}\varphi_n} \ar[rrr]^{\Sigma^{-1}f_n} \ar@{->>}[rd]^{\pi}  &  &  &  X_1 \ar[rr]^{f_1} \ar[ddd]^{\varphi_1}\ar[lllddd]^{h_1}\ar[ldd]^{c}\ar@{->>}[rd]^{\pi_1} & & X_2 \ar[ddd]^{\varphi_2}\ar[llddd]^{h_2} \ar[r]^{f_2}& \cdots \ar[r]^{f_{n-1}} & X_n \ar[r]^{f_n}\ar[ddd]^{\varphi_n} & \Sigma X_1\ar[ddd]^{\Sigma\varphi_1} \ar[lddd]^{h_{n+1}}\\
 & M \ar[ddd]^{h} \ar@{>->}[urr]^{l}\ar[rd]^a & & & M_1 \ar@{>->}[ru]^{l_1}\ar[ldd]_{m_1} &  & & & & \\
&  &  I \ar[ldd]^b \ar[lld]^{d} & & & & & &\\
\Sigma^{-1}Y_n \ar[rrr]^{\Sigma^{-1}g_n}\ar@{->>}[rd]^{\pi'} & & & Y_1 \ar[rr]^{g_1} & & Y_2 \ar[r]^{g_2} & \cdots \ar[r]^{g_{n-1}} & Y_n \ar[r]^{g_n} &\Sigma Y_1\\
&  N \ar@{>->}[urr]^{l'}& & & & & & &\\
}$$

Let $\varphi_\bullet:(X_\bullet,f_\bullet)\rightarrow (Y_\bullet,g_\bullet)$ and $\psi_\bullet: (Y_\bullet,g_\bullet)\rightarrow (Z_\bullet,h_\bullet)$ be morphisms in the kernel of $Z_1$. Up to homotopy, we assume that $\varphi_\bullet=(0,0,\cdots,0,\varphi_n)$ and $\psi_\bullet=(0,0,\cdots,0,\psi_n)$. Thus we get the following diagram.
$$
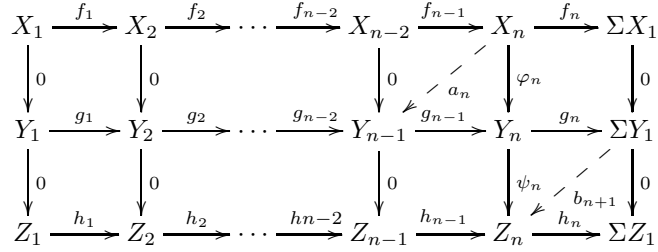
\xymatrix{
X_1 \ar[r]^{f_1}\ar[d]^{0} & X_2 \ar[r]^{f_2}\ar[d]^{0} & \cdots \ar[r]^{f_{n-2}}& X_{n-2} \ar[r]^{f_{n-1}}\ar[d]^{0}& X_n \ar[r]^{f_n}\ar[d]^{\varphi_n}\ar@{-->}[ld]^{a_n} & \Sigma X_1 \ar[d]^{0}\\
Y_1 \ar[r]^{g_1}\ar[d]^0 & Y_2 \ar[r]^{g_2} \ar[d]^0 & \cdots \ar[r]^{g_{n-2}}& Y_{n-1} \ar[r]^{g_{n-1}}\ar[d]^0 & Y_n \ar[r]^{g_n}\ar[d]^{\psi_n}& \Sigma Y_1\ar[d]^0\ar@{-->}[ld]^{b_{n+1}}\\
Z_1\ar[r]^{h_1} & Z_2 \ar[r]^{h_2}  & \cdots \ar[r]^{h{n-2}} & Z_{n-1} \ar[r]^{h_{n-1}} & Z_n \ar[r]^{h_n} & \Sigma Z_1 \\
}$$
Since $g_n\varphi_n=0$ and $\psi_ng_{n-1}=0$, we have $\varphi_n$ factors through $g_{n-1}$ and $\psi_n$ factors through $g_n$. Thus $\psi_n\varphi_n=b_{n+1}g_ng_{n-1}a_n=0$. So $\psi_\bullet\varphi_\bullet=0$.
\end{proof}

The following proposition is a higher version of \cite[Proposition 8.7]{[Am]}.

\begin{prop}\label{2.5}
The category of $\Phi$-$n$-$\Sigma$-complexes is equivalent to the category of $n$-$\Sigma$-complexes which are homotopy-equivalent to standard $n$-angles.
\end{prop}

\begin{proof}
Since standard $n$-angles are $\Phi$-$n$-$\Sigma$-complexes, Lemma \ref{2.2} implies that each  $n$-$\Sigma$-complex which is homotopy-equivalent to a standard $n$-angle is a $\Phi$-$n$-$\Sigma$-complex. Let $(X_\bullet, f_\bullet)$ be a $\Phi$-$n$-$\Sigma$-complex. Let $M$ be the kernel of $f_1$. Since $X^1M$ and $X^2M$ are projective-injective, we can find morphisms $\varphi_1:X_1\rightarrow X^1M$ and $\varphi_2:X_2\rightarrow X^2M$ such that the following diagram is commutative.
$$\xymatrix{
& X_1 \ar[r]^{f_1}\ar[dd]^{\varphi_1} & X_2 \ar[r]^{f_2}\ar[dd]^{\varphi_2} & X_3 \ar[r]^{f_3} & \cdots \ar[r]^{f_{n-1}}& X_n \ar[r]^{f_n} & \Sigma X_1 \ar[dd]^{\Sigma \varphi_1}\\
 M \ar[ru]\ar@{=}[dd] & & & & & & & \\
 &  X^1M \ar[r]^{} & X^2M \ar[r]^{} & X^3M \ar[r]^{} & \cdots \ar[r]^{} & X^nM \ar[r]^{}& \Sigma X^1M\\
 M\ar[ru] & & & & & & & & &\\
}$$ We can complete $(\varphi_1,\varphi_2)$ to an $n$-$\Sigma$-periodic morphism $\varphi_\bullet=(\varphi_1,\varphi_2, \cdots,\varphi_n)$ from $X_\bullet$ to $T_M$ by Lemma \ref{2.3}. Since $Z_1(\varphi_\bullet)=\mbox{Id}_M$, we obtain that $\varphi_\bullet$ is a homotopy-equivalence by Lemma \ref{2.4}, i.e., $X_\bullet$ is homotopy-equivalent to $T_M$.
\end{proof}

\section{proof of Theorem 1.1}

\noindent {\bf Proof of Theorem 1.1}. We are going to show that $(\mathcal{C},\Sigma,\Phi)$ is an $n$-angulated category.

(N1a) and (N1b) are trivial.

(N1c). Let $f_1:X_1\rightarrow X_2$ be a morphism in $\mathcal{C}$, $A=\mbox{ker}f_1$ and $B=\mbox{coker}f_1$. By sequence (2.1) we easily obtain the following commutative diagram
$$\xymatrix{
0 \ar[r] & A \ar[r]^l\ar@{=}[d] & X_1 \ar[r]^{f_1}\ar[d] & X_2 \ar[r]^{f_2}\ar[d] & X^1B \ar[r]^{f_3} \ar[d]&  \cdots \ar[r]^{f_{n-2}} & X^{n-3}B \ar[r]^{\pi_{n-1}}\ar[d] & C \ar[d]^g \ar[r]& 0\\
0 \ar[r] & A \ar[r] & I_A \ar[r] & I_{\Omega^{-1}A} \ar[r] & I_{\Omega^{-2}A} \ar[r] & \cdots \ar[r] & I_{\Omega^{2-n}A} \ar[r] & \Omega^{1-n}A \ar[r] & 0\\
}$$
with exact rows. Since $g$ is an isomorphism in $\underline{\mbox{mod}}\mathcal{C}$, we take $h=(\Omega^{-1}g)^{-1}\alpha_A$. Consider the following commutative diagram
$$\xymatrix{
0\ar[r] & C \ar[r]^{l_{n-1}}\ar@{=}[d] & X_n \ar[r]^{\pi_n} \ar[d] & \Sigma A\ar[r]\ar[d]^h & 0 \\
0\ar[r] & C \ar[r]^{i_C}\ar[d]^g & I_C \ar[r]^{p_C}\ar[d] & \Omega^{-1}C \ar[r]\ar[d]^{\Omega^{-1}g} & 0\\
0\ar[r] & \Omega^{1-n}A \ar[r] & I_{\Omega^{1-n}A} \ar[r] & \Omega^{-n}A \ar[r] & 0\\
}$$ where $X_n$ is the pullback of $h$ and $p_C$.
It is easy to see that $$X_1\xrightarrow{f_1}X_2\xrightarrow{f_2} X^1B\xrightarrow{f_3} X^2B\xrightarrow{f_{4}}\cdots\xrightarrow{f_{n-2}} X^{n-3}B\xrightarrow{l_{n-1}\pi_{n-1}} X_n\xrightarrow{\Sigma l\cdot\pi_n}\Sigma X_1$$
is a $\Phi$-$n$-$\Sigma$-complex.

(N2). Let $X_\bullet$ be a $\Phi$-$n$-$\Sigma$-complex. Since $X_\bullet[1]$ is isomorphic to the left rotation of $X_\bullet$ and $X_\bullet[-1]$ is isomorphic to the right rotation of $X_\bullet$, we only need to show that $X_\bullet[1]$ and $X_\bullet[-1]$ are $\Phi$-$n$-$\Sigma$-complexes. In fact,  $X_\bullet$ is homotopy-equivalent to $T_M$ for some object $M\in\mbox{mod}\mathcal{C}$ by Proposition \ref{2.5}. Thus $X_\bullet[1]$ is homotopy-equivalent to $T_M[1]$. Since $T:\underline{\mbox{mod}}\mathcal{C}\rightarrow K^{\tiny\mbox{ex}}_{n\mbox{-}\Sigma}(\mbox{proj}\mathcal{C})$ is a triangle functor, we obtain that $T_{\Omega^{-1}M}$ is isomorphic to $T_M[1]$. Now $X_\bullet[1]$ is homotopy-equivalent to $T_{\Omega^{-1}M}$, which implies that $X_\bullet[1]$ is a  $\Phi$-$n$-$\Sigma$-complex. Similarly we can show $X_\bullet[-1]$ is a $\Phi$-$n$-$\Sigma$-complex.

(N3). It follows from Lemma \ref{2.3}.

(N4). Suppose we have a commutative diagram $$\xymatrix{
X_\bullet:& X_1 \ar[r]^{f_1}\ar[d]^{\varphi_1} & X_2 \ar[r]^{f_2}\ar[d]^{\varphi_2} & X_3 \ar[r]^{f_3} & \cdots \ar[r]^{f_{n-1}}& X_n \ar[r]^{f_n} & \Sigma X_1 \ar[d]^{\Sigma \varphi_1}\\
Y_\bullet:& Y_1 \ar[r]^{g_1} & Y_2 \ar[r]^{g_2} & Y_3 \ar[r]^{g_3} & \cdots \ar[r]^{g_{n-1}} & Y_n \ar[r]^{g_n}& \Sigma Y_1\\
}$$ whose rows are $\Phi$-$n$-$\Sigma$-complexes. Let $(M,l:M\rightarrow X_1)$ be the kernel of $f_1$, $(N,l':N\rightarrow Y_1)$ be the kernel of $g_1$ and $h:M\rightarrow N$ the induced morphism. Then there exist two homotopy-equivalences $a_\bullet:X_\bullet\rightarrow T_M$ and $b_\bullet:T_N\rightarrow Y_\bullet$ by the proof of Proposition \ref{2.5}. Let $\phi_\bullet=b_\bullet\cdot T(h)\cdot a_\bullet=(\phi_1,\phi_2,\cdots,\phi_n)$ be the morphism from $X_\bullet$ to $Y_\bullet$. It is easy to see that $(\varphi_1-\phi_1)l=0$, so there exists a morphism $h_2:X_2\rightarrow Y_1$ such that $\varphi_1-\phi_1=h_2f_1$. Note that $(\varphi_2-\phi_2-g_1h_2)f_1=0$, there exists a morphism $h_3:X_3\rightarrow Y_2$ such that $\varphi_2-\phi_2-g_1h_2=h_3f_2$, i.e., $\varphi_2-\phi_2=g_1h_2+h_3f_2$. Let $\varphi_3=\phi_3+g_2h_3$, then $g_3\varphi_3=g_3\phi_3=\phi_4f_3$. If we take $\varphi_4=\phi_4,\cdots,\varphi_n=\phi_n$, then $g_i\varphi_i=\varphi_{i+1}f_i$, $i=4,\cdots,n-1$, and $g_n\varphi_{n}=g_n\phi_n=\Sigma\phi_1\cdot f_n=\Sigma(\varphi_1-h_2f_1)\cdot f_n=\Sigma\varphi_1\cdot f_n$. Thus $\varphi_\bullet=(\varphi_1,\varphi_2,\varphi_3,\cdots,\varphi_n)$ is an $n$-$\Sigma$-periodic morphism and $\varphi_\bullet$ is $n$-$\Sigma$-homotopic to $\phi_\bullet$ with the homotopy $(0, h_2,h_3,0,\cdots,0)$.

It remains to show that the cone $C(\varphi_\bullet)$ is a $\Phi$-$n$-$\Sigma$-complex. In fact, since $\varphi_\bullet$ is $n$-$\Sigma$-homotopic to $\phi_\bullet=b_\bullet\cdot T(h)\cdot a_\bullet$, where $a_\bullet$ and $b_\bullet$ are homotopy-equivalent, we obtain that the cones $C(\varphi_\bullet)$, $C(\phi_\bullet)$ and $C(T(h))$ are isomorphisms in $K^{\tiny\mbox{ex}}_{n\mbox{-}\Sigma}(\mbox{proj}\mathcal{C})$. Let $M\xrightarrow{h}N\rightarrow C(h)\rightarrow \Omega^{-1}M$ be a triangle in $\underline{\mbox{mod}}\mathcal{C}$. Since $T:\underline{\mbox{mod}}\mathcal{C}\rightarrow K^{\tiny\mbox{ex}}_{n\mbox{-}\Sigma}(\mbox{proj}\mathcal{C})$ is a triangle functor, $T_M\xrightarrow{T(h)}T_N\rightarrow T_{C(h)}\rightarrow T_M[1]$ is a triangle in $K^{\tiny\mbox{ex}}_{n\mbox{-}\Sigma}(\mbox{proj}\mathcal{C})$. Thus $C(T(h))\cong T_{C(h)}$ in $K^{\tiny\mbox{ex}}_{n\mbox{-}\Sigma}(\mbox{proj}\mathcal{C})$. By these isomorphisms and Proposition \ref{2.5} we get $C(\varphi_\bullet)$ is a $\Phi$-$n$-$\Sigma$-complex. \hspace{4.2cm} $\Box$

\section{Application to self-injective algebrs}

In this section, we will apply Theorem \ref{1} to self-injective algebras and give some examples.

\begin{lem}(\cite[Lemma 1.5]{[GSS]})\label{4.1}
Let $A$ be a finite-dimensional indecomposable quasi-periodic $k$-algebra, then $A$ is a self-injective algebra.
\end{lem}

\begin{lem}\label{4.2}
Let $A$ be a finite-dimensional self-injective $k$-algebra. If there exists an exact sequence of $A$-$A$-bimodules
$$\begin{gathered} 0\rightarrow\ _1A_\sigma \rightarrow P_n\rightarrow P_{n-1}\rightarrow \cdots\rightarrow P_1\rightarrow A\rightarrow 0
\end{gathered} \eqno (4.1) $$
where $\sigma$ is an automorphism of $A$ and the $P_i$'s  are projective as bimodules, then $\mbox{proj}A$ has an $n$-angulation structure where the suspension functor is $-\otimes_AA_{\sigma^{-1}}$.
\end{lem}

\begin{proof} Since $A$ is self-injective, mod$A$ is a Frobenius category. If one tensors the  sequence (4.1) with $_1A_{\sigma^{-1}}$, one obtain the following exact sequence of $A$-$A$-bimodules
$$
0\rightarrow A \rightarrow\ _1A_{\sigma^{-1}}\otimes_A P_n\rightarrow\ _1A_{\sigma^{-1}}\otimes_AP_{n-1}\rightarrow \cdots\rightarrow\ _1A_{\sigma^{-1}}\otimes_AP_1\rightarrow\ _1A_{\sigma^{-1}} \rightarrow 0
 $$ where all the $_1A_{\sigma^{-1}}\otimes_AP_i$ are projective as bimodules. Thus we have the following exact sequence of exact endofunctors of mod$A$
$$0\rightarrow \mbox{Id} \rightarrow -\otimes_AA_{\sigma^{-1}}\otimes_A P_n\rightarrow  -\otimes_AA_{\sigma^{-1}}\otimes_AP_{n-1}\rightarrow \cdots$$
$$\rightarrow -\otimes_AA_{\sigma^{-1}}\otimes_AP_1\rightarrow  -\otimes_AA_{\sigma^{-1}} \rightarrow 0.$$
Moreover, the functors $-\otimes_AA_{\sigma^{-1}}\otimes_AP_i$ take values in proj$A$. Note that $e_iA\otimes_AA_{\sigma^{-1}}\cong e_iA_{\sigma^{-1}}\cong\sigma^{-1}(e_i)A$ for each idempotent $e_i$ of $A$, the functor $-\otimes_AA_{\sigma^{-1}}: \mbox{proj}A\rightarrow\mbox{proj}A$ is an automorphism.
 By Theorem \ref{1}, we deduce that
$(\mbox{proj}A, -\otimes_AA_{\sigma^{-1}})$ has an $n$-angulation.
\end{proof}


\begin{thm}
Let $A$ be a finite-dimensional indecomposable quasi-periodic $k$-algebra. Assume that $\Omega^n_{A^e}(A)\cong$ $_1A_\sigma$ as $A$-$A$-bimodules for
  an automorphism $\sigma$ of $A$. Then for each positive integer $m$,  the category $(\mbox{proj}A, \Sigma)$ has an $mn$-angulation structure, where $\Sigma$ is the functor $-\otimes_AA_{\sigma^{-m}}:\mbox{proj}A\rightarrow\mbox{proj}A$.
  In particular, if $\sigma$ is of finite order $l$, then $(\mbox{proj}A, \mbox{Id}_{\mbox{proj}A})$ has an $ln$-angulation structure.
\end{thm}

\begin{proof}
By Lemma \ref{4.1}, we know that $A$ is a self-injective algebra.
We claim that there exists an exact sequence of $A$-$A$-bimodules
$$\begin{gathered}
 0\rightarrow\ _1A_{\sigma^{m}} \rightarrow P_{mn}\rightarrow P_{mn-1}\rightarrow \cdots\rightarrow P_{n+1}\rightarrow P_{n}\rightarrow\cdots\rightarrow P_1\rightarrow A\rightarrow 0
\end{gathered}$$
where the $P_i$'s  are projective as bimodules. Thus the theorem immediately follows from Lemma 4.2. We prove this claim by induction on $m$. Since $\Omega^n_{A^e}(A)\cong$ $_1A_\sigma$ in mod$A^e$, there exists an exact sequence (4.1), 
where the $P_i$'s  are projective as $A$-$A$-bimodules.
Assume now that $m>1$ and our claim holds for $m-1$.
Applying the functor $_{1}A_{\sigma^{m-1}}\otimes_A-$ to the  sequence (4.1), we obtain the following exact sequence of $A$-$A$-bimodules
$$0\rightarrow\ _{1}A_{\sigma^m}\rightarrow\ _{1}A_{\sigma^{m-1}}\otimes_AP_n\rightarrow\cdots\rightarrow\  _{1}A_{\sigma^{m-1}}\otimes_AP_1\rightarrow\ _{1}A_{\sigma^{m-1}}\rightarrow 0. \ \eqno (4.2)$$
Take $P_{(m-1)n+i}=\ _{1}A_{\sigma^{m-1}}\otimes_AP_i$, where $i=1,2,\cdots,n$. Then the $P_{(m-1)n+i}$'s are projective as bimodules. By induction and (4.2),
we obtain that our claim holds for each positive integer $m$.
\end{proof}

In particular, we get the following easy corollary.

\begin{cor}
Let $A$ be a finite-dimensional periodic $k$-algebra of periodicity $n$. Then the category $(\mbox{proj}A, \mbox{Id}_{\mbox{proj}A})$ has an $n$-angulation structure. Moreover, for each positive integer $m$, the category $(\mbox{proj}A, \mbox{Id}_{\mbox{proj}A})$ has an $mn$-angulation structure.
\end{cor}

\begin{example} We will revisit \cite[Corollary 9.3]{[Am]}.
Let $\Delta$ be a graph of generalized Dynkin type and $A=P^f(\Delta)$ be the corresponding deformed preprojective algebra introduced by Bialkowski-Erdmann-Skowro\'{n}ski \cite{[BES]}. We note that if $f$ is zero, then $P^f(\Delta)$ is just the usual preprojective algebra introduced by Gelfand-Ponomarev \cite{[GP]}.
By \cite[Proposition 3.4]{[BES]}, we get $\Omega^3_{A^e}(A)\cong$ $_{1}A_{\sigma^{-1}}$ as $A$-$A$-bimodules for an automorphism $\sigma$ of $A$ of finite order. Moreover, for each idempotent $e_i$ of $A$, we have $\sigma(e_i)=e_{\nu(i)}$, where $\nu$ is the Nakayama permutation.
By Theorem 4.1, proj$A$ is a triangulated category, and the suspension functor $-\otimes_A A_{\sigma}$ turns out to be the Nakayama functor.
Let $m$ be the order of $\sigma$, then $(\mbox{proj}A, \mbox{Id}_{\tiny\mbox{proj}A})$ has a $3m$-angulation structure.
\end{example}

\begin{example}
Let $A=kQ_n/I_s$ be a self-injective Nakayama $k$-algebra, where $n\geq1$, $s\geq2$, $Q_n$ is the quiver
 $$\xymatrixcolsep{1.5pc}
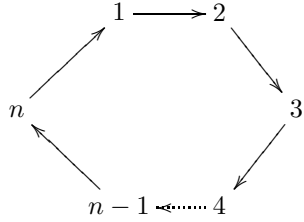
\xymatrix{
 & 1\ar[r]& 2\ar[rd]& \\
 n\ar[ru] & & & 3\ar[dl] \\
 & {\small n-1}\ar[lu]&  4 \ar@{..>}[l] & \\
 }$$ and $I_s$ is the ideal generated by paths of length $s$. It is easy to see that $A$ is  of finite representation type. In the notation of Asashiba this is of type $(A_n, \frac{s}{n},1)$. By Table 5.2 in \cite{[D2]}, we know the periodicity of $A$ is $$p=\left\{
                   \begin{array}{cc}
                     s ,& k=2, n=1\ \mbox{and}\ 2\nmid s; \\
                     \frac{2s}{(s,n+1)}, & \mbox{otherwise}. \\
                   \end{array}
                 \right.
 $$ Thus $(\mbox{proj}A, \mbox{Id}_{\mbox{proj}A})$ has a structure of $p$-angulated category.
\end{example}


\vspace{2mm}\noindent {\bf Acknowledgements} \ The author thanks  professor Xiaowu Chen for drawing his attention to the reference \cite{[Am]} and for valuable conversation on this topic.

\end{document}